\documentclass[a4paper,11pt,reqno]{amsart}
\usepackage{amsmath}
\usepackage{amssymb}
\usepackage{xcolor}
\usepackage{enumerate}
\usepackage{soul}

\usepackage{mathrsfs}

\theoremstyle{plain} 
\newtheorem{theorem}{Theorem}[section]

\newtheorem{lemma}[theorem]{Lemma}

\newtheorem{remark}[theorem]{Remark}

\theoremstyle{plain}
\newtheorem*{theoremN}{Theorem}

\makeatletter
    
    \@addtoreset{equation}{section}
  \makeatother

\title[Elephant random walk with a power law memory]{Phase transitions for a unidirectional elephant random walk with a power law memory II: Some sharper estimates} 
\author{Rahul Roy}
\address{Indraprastha Institute of Information Technology and Indian Statistical Institute, Delhi, India}
\email{rahul.roy@iiitd.ac.in}
\author{Masato Takei}
\address{Department of Applied Mathematics, Faculty of Engineering, Yokohama National University, Yokohama, Japan}
\email{takei-masato-fx@ynu.ac.jp}
\author{Hideki Tanemura}
\address{Department of Mathematics, Keio University, Yokohama, Japan}
\email{tanemura@math.keio.ac.jp}

\usepackage{xcolor}

\usepackage{epsfig}

 \usepackage{pst-all}

\usepackage{graphicx}

\begin{document}

\begin{abstract}
We continue our study of the unidirectional elephant random walk (uERW) initiated in {\it {Electron. Commun. Probab.}} ({\bf 29} 2024, article no. 78). In this paper we obtain definitive results when  the memory exponent $\beta\in (-1, p/(1-p))$. In particular using a coupling argument we obtain the exact asymptotic rate of growth of $S_n$, the location of the uERW at time $n$, for the case $\beta\in (-1, 0] $. Also, for the case $\beta\in (0, p/(1-p))$ we show that  $P(S_n \to \infty) \in (0,1)$ and conditional on $\{S_n \to \infty\}$ we obtain the exact asymptotic rate of growth of $S_n$. In addition we obtain the central limit theorem for $S_n$ when $\beta \in (-1, p/(1-p))$.
\end{abstract}

\maketitle

\section{Introduction}
\label{sec:intro}

In \cite{RoyTakeiTanemura24ECP} we had studied a version of unidirectional elephant random walk (uERW) introduced by Harbola {\it et al.} \cite{HarbolaEtal14PRE}. In this model, let $\{\beta_{n+1}: n \in \mathbb{N} \}$ be a sequence of  independent random variables with 
\begin{align}
\label{r-beta1}
P(\beta_{n+1} = k) = \begin{cases} \dfrac{\beta + 1}{n} \cdot \dfrac{\mu_k}{\mu_{n+1}} & \text { for } 1 \leq k \leq n\\
0  & \text{ otherwise}
\end{cases}
\end{align}
where, $\beta > -1$ and
\begin{equation}
\label{r-mu1}
\mu_n = \frac{\Gamma(n+ \beta)}{\Gamma(n)\Gamma( \beta +1)} \sim \frac{n^\beta}{\Gamma(\beta +1)} \quad \mbox{as $n \to \infty$.}
\end{equation}
The unidirectional ERW $\{S_n : n \geq 0\}$ is given by $S_0=0$ and, for $n \geq 1$,
\begin{align}
\label{r-lazydef2}
S _n := \sum_{k=1}^n X_k \text{ with } 
 X_1 \equiv 1,\, X_{n+1} := \begin{cases}
X_{\beta_{n+1}} & \text{with probability }p\\
0 &  \text{with probability }1-p.
\end{cases}
\end{align}
Let $\Sigma_n := \sum_{k=1}^n \mu_k X_k$ for $n \in \mathbb{N}$. In \cite{RoyTakeiTanemura24ECP} it was noted that,  for $\gamma > -1$, taking
$c_n(\gamma) := \dfrac{\Gamma(n+\gamma)}{\Gamma(n)\Gamma(\gamma+1)}$, the process
$\{M_n: n \in \mathbb{N}\}$, where
\begin{align}
\label{def:MnMart}
M_n := \dfrac{\Sigma_n}{c_n(p(\beta+1))},
\end{align}
is a non-negative martingale with 
$M_{\infty}= \displaystyle \lim_{n \to \infty}M_n$
existing almost surely.

In \cite{RoyTakeiTanemura24ECP} we had obtained different behaviour on the asymptotics of $S_n$ depending on the value of $\beta$.
These results are summarized in Table \ref{table1}. 
In both Table \ref{table1} and Table \ref{table2} we take 
\begin{align}
\Omega_{\infty}(p,\beta):= \{ \mbox{$M_{\infty}>0,\,S_n \sim   C(p,\beta) M_{\infty} n^{p (\beta+1)-\beta}$ as $n \to \infty$} \},
\end{align}
and
\begin{align}
C(p,\beta):=\dfrac{1}{p(\beta+1)-\beta} \cdot \dfrac{\Gamma(\beta+1)}{\Gamma(p(\beta+1))}. 
\label{eq:DefCbeta}
\end{align}
Here, and elsewhere, $a_n \sim b_n$, $n\to \infty$ means $a_n / b_n \to 1$, $n\to\infty$. 
%

\begin{table}[ht]
\centering
\begin{tabular}{|c||c|c|} \hline
 Regime & Asymptotic behaviour \\ \hline \hline
$-1<\beta<0$ & $P(S_{\infty}=+\infty)=1$, $P(\Omega_{\infty}(p,\beta))>0$. 
\\ \hline
$\beta=0$ & 
 $P(\Omega_{\infty}(p,0))=1$. \\ \hline
$0<\beta<\dfrac{p}{1-p}$ &  $0< P(S_{\infty} = + \infty)<1$,
$P\left(
\Omega_{\infty} (p,\beta) \mid S_{\infty} = +\infty \right)>0$. 
\\ \hline
$\beta = \dfrac{p}{1-p}$ & $E[S_n] \sim \beta \log n$, but  $P(S_{\infty} < + \infty)=1$. \\ \hline
$\beta > \dfrac{p}{1-p}$ & $E[S_{\infty}] < +\infty$, so $P(S_{\infty} < + \infty)=1$. \\ \hline
\end{tabular}
\medskip
\caption{Summary of the results obtained in \cite{RoyTakeiTanemura24ECP}.} \label{table1}
\end{table}

In this note we obtain some sharper estimates regarding the martingale sequence $\{M_n : n\geq 0\}$ for $-1 < \beta < p/(1-p)$. This allows us to have a definitive understanding of the asymptotic behaviour of $S_n$  in the different regimes
as presented in Table \ref{table2}. 

\begin{table}[ht]
\centering
\begin{tabular}{|c||c|c|} \hline
 Regime & Asymptotic behaviour \\ \hline \hline
$-1<\beta\leq 0$ & $P\left(
\Omega_{\infty}(p,\beta) \right)=1$. \\ \hline
$0<\beta<\dfrac{p}{1-p}$ &  $0 < P(S_{\infty} = + \infty ) <1$, 
$P\left(\Omega_{\infty} (p,\beta) \mid S_{\infty} = +\infty \right)=1$.\\ \hline
$\beta = \dfrac{p}{1-p}$ & $E[S_n] \sim \beta \log n$, but  $P(S_{\infty} < + \infty)=1$. \\ \hline
$\beta > \dfrac{p}{1-p}$ & $E[S_{\infty}] < +\infty$, so $P(S_{\infty} < + \infty)=1$. \\ \hline
\end{tabular}
\medskip
\caption{Behaviour of $S_n$ in different regimes.}\label{table2}
\end{table}

Moreover, taking 
\begin{align}
W_n := S_n - C(p,\beta) M_{\infty} n^{p(\beta+1)-\beta} ,
\end{align}
we obtain the central limit theorem for $\{W_n\}$ in different regimes of $\beta$.
Let $\eta$ be a non-negative random variable defined by
 \begin{align}
\eta =\sqrt{\frac{p^2(\beta+1)^2+\beta^2}{(p(\beta+1)-\beta)^2} \cdot  C(p,\beta) \cdot M_\infty}.
\label{eq:CLTlimitingV}
 \end{align}

\begin{theorem}\label{:Theorem 3.1}  Assume that $p \in (0,1)$ and let $N \stackrel{d}{=} N(0,1)$. \\
(i) If $\beta \in (-1,p/(1-p))$ then $\dfrac{W_n}{\sqrt{n^{p(\beta+1)-\beta}}} \stackrel{d}{\to}  \eta\cdot N$ as $n\to\infty$,
where $N$ is independent of $\eta$. 

\noindent
(ii) If $\beta \in (-1, 0]$ then $P(\eta >0) =1$ and $\dfrac{W_n}{\eta \sqrt{n^{p(\beta+1)-\beta}}} \stackrel{d}{\to}  N$ as $n\to\infty$. 
If $\beta \in (0, p/(1-p))$ then $\{ \eta >0\} =\{ S_\infty=\infty\}$ a.s., and $\dfrac{W_n}{\eta \sqrt{n^{p(\beta+1)-\beta}}}
\stackrel{d}{\to}  N$ as $n\to\infty$ under $P(\,\cdot \, \mid S_\infty=\infty)$. 
\end{theorem}

\begin{remark}
The case $\beta=0$ was obtained by Miyazaki and Takei \cite{MiyazakiTakei20JSP}, based on the ideas in Kubota and Takei \cite{KubotaTakei19JSP}.
\end{remark}

In the next section we introduce an auxiliary process. 
In Section \ref{sec:Postivity} we have a coupling argument, however not with a branching process as in \cite{RoyTakeiTanemura24ECP}, and we prove the results displayed in Table \ref{table2}. 
In Section \ref{sec:CLT} we prove Theorem \ref{:Theorem 3.1}. These results together with the limit laws for $S_n$ obtained in \cite{RoyTakeiTanemura24ECP} for the other cases of $\beta$ complete our understanding of the uERW with a power law memory.


\section{A modified version of the process}

Let $\{x_n\}_{n\in\mathbb{N}}$ be a sequence of $0$'s and $1$'s, and let 
\begin{align} \label{def:Subset}
 \mathbb{S}:=\{k\in\mathbb{N} : x_k=1  \}. 
\end{align}
Let $1 \le s_1 < s_2 < \cdots $ be the ordering of all elements of $\mathbb{S}$ and
\begin{align}\label{;def_mn}
m_n=m_n(\mathbb{S}):=\# \{k\in\mathbb{N} : s_1 < k \le n , x_k=0  \}. 
\end{align}
We assume that $\{x_n\}_{n\in\mathbb{N}}$ satisfies the following:
there exists  $N_0=N_0(\mathbb{S})$ such that
\begin{align}\label{;condA}
m_n \le  n^{p(\beta+1)-\beta} \quad \mbox{for all $n\ge N_0$.}
\end{align}
Note that ${p(\beta+1)-\beta} \in (0,1)$ for $\beta \in (-1, p/(1-p))$.

We introduce a modified version of the process.
Let $\{ \tilde{\beta}_{n+1}: n\in\mathbb{N} \}$ be a collection of independent random variables on the same probability space as earlier, but with a probability measure $P^{\mathbb{S}}$ given by
\begin{align}\label{;beta_t}
P^{\mathbb{S}}(\tilde{\beta}_{n+1}=k) =
\begin{cases}
\displaystyle{ w(n,k):=\frac{x_k \mu_k}{\sum_{\ell=1}^n \mu_\ell} }\quad &\text{for $1\le k \le n$}
\\
1-\sum_{\ell=1}^n w(n,\ell) &\text{for $k=0$}
\\
0 &\text{otherwise.}
\end{cases}
\end{align}
For $s_1=k \in \mathbb{N}$, let
\begin{align*}
&\text{$Y_{\ell}=0$ for $0 \le \ell \le k-1$,  $Y_k = 1$, and}
\\
&\text{for $n \ge  k$, }Y_{n+1} = \begin{cases}
x_{n+1}Y_{\tilde{\beta}_{n+1}} &\text{with probability $p$}
\\
0 &\text{with probability $1-p$.}
\end{cases}
\end{align*}
The modified models for $S_n$ and $\Sigma_n$ are given by $T_0=\Xi_0=0$,
\begin{align}
T_n := \sum_{k=1}^n Y_k \quad \text{and} \quad 
\Xi_n := \sum_{k=1}^n \mu_k Y_k,
\ n\in\mathbb{N},
\end{align}
respectively. 

\noindent {\bf NOTATION:} For any sequences $\{a_n\}, \{b_n\}$,
\begin{itemize}
\item $a_n \asymp b_n$, $n\to \infty$ means that $c a_n \le b_n \leq Ca_n$ for some $0<c \le C<\infty$,
\item $a_n \simeq b_n$ means that $\{ |a_n -b_n| \} \le C$, $n\in\mathbb{N} $ for some $C>0$.
\end{itemize}
Hereafter the constants $c$ and $C$ may depend on $\beta > -1$, $p\in (0,1)$, and $\mathbb{S}$. 
\begin{lemma} \label{thm:XiSigmaMeanGen} 
Let $p\in (0,1)$, $\beta \in (-1,p/(1-p))$, and $\mathbb{S}$ be as in \eqref{def:Subset}.
Under the condition \eqref{;condA},
\begin{align*}
E^{\mathbb{S}}[\Xi_n] \asymp n^{p(\beta+1)},
\quad n\to\infty.
\end{align*}
\end{lemma}
\begin{proof}
Let $\mathcal{F}_n$ be the $\sigma$-algebra generated by $Y_1,\ldots,Y_n$.
Recalling that
\begin{align}\label{;sum_mu}
\sum_{\ell=1}^n \mu_\ell = \frac{n}{\beta+1}\mu_{n+1}= c_n(\beta+1),
\end{align}
and $w_n:= \dfrac{x_{n+1}\mu_{n+1}}{\sum_{\ell=1}^n \mu_\ell} =\dfrac{(\beta+1)x_{n+1}}{n}$,  
we see that
\begin{align*}
E^{\mathbb{S}}[Y_{n+1} \mid \mathcal{F}_n] 
= px_{n+1}\cdot  E^{\mathbb{S}}[ Y_{\tilde{\beta}_{n+1}} \mid \mathcal{F}_n]
&=p x_{n+1} \sum_{k=1}^n  \dfrac{\mu_k}{\sum_{\ell=1}^n \mu_\ell}Y_k 
=\dfrac{p w_n}{\mu_{n+1}} \cdot \Xi_n.
\end{align*}
Noting that $Y_k=0$ if $x_k=0$, the above holds irrespective of whether $n+1$ is smaller or larger than $s_1$.
Then
$E^{\mathbb{S}}[\Xi_{n+1} \mid \mathcal{F}_n] = (1+pw_n) \Xi_n$, 
and
\begin{align}\label{;2g}
E^{\mathbb{S}}[\Xi_{n+1}] &=  (1+pw_n)E^{\mathbb{S}}[\Xi_n]
=\mu_{s_1}\prod_{k=s_1}^n \left(1+pw_k\right).
\end{align}

In the special case $\mathbb{S}=\mathbb{N}$, 
we have $w_k=\check{w}_k:= \dfrac{\beta+1}{k}$, and 
$\log E[\Sigma_{n+1}] \simeq p(\beta+1) \log n$, $n\to\infty$.
We now show that for any $\mathbb{S}$ satisfying the condition \eqref{;condA},
\begin{align}\label{;0b}
\log E^{\mathbb{S}}[\Xi_{n+1}] \simeq p(\beta+1) \log n, \quad n\to\infty.
\end{align}
To prove this, in view of \eqref{;2g}, it is enough to show that under \eqref{;condA},
\begin{align} \label{;1a+}
\sum_{k=s_1}^n  \log\left( 1 + pw_k \right) \simeq p \sum_{k=s_1}^n \check{w}_k.
\end{align}
Using $w_k \le \check{w}_k$ and $x-\dfrac{x^2}{2} \le \log(1+x) \le x$ for $x \ge 0$, we have
\begin{align} \label{;1a++}
0 \le p\check{w}_k - pw_k \le p\check{w}_k - \log(1+pw_k) \le  p\check{w}_k - pw_k+\dfrac{(pw_k)^2}{2}.
\end{align}
Since $\beta+1>0$, 
\begin{align}\label{;1b}
\sum_{k=1}^{\infty} w_k^2  \leq \sum_{k=1}^{\infty} \frac{(\beta+1)^2}{k^2} <\infty.
\end{align}
Now we estimate $\displaystyle \sum_{k=s_1}^n (\check{w}_k - w_k) = \sum_{k=s_1}^n \dfrac{1-x_{k+1}}{k}$.
Let $u_1 < u_2 < \cdots$ be the ordering of all elements of $\{k \in \mathbb{N} : s_1<k,x_{k+1}=0\}$.
By \eqref{;condA},
\[ J_n := \# \{k \in \mathbb{N} : s_1<k \le n ,x_{k+1}=0\} \le m_{n+1} \le (n+1)^{p(\beta+1)-\beta} \mbox{ for $n \ge N_0$.} \]
Noting that $u_j < n$ implies $j \le J_n \le (n+1)^{p(\beta+1)-\beta}$, for all $j \ge N_0$,
\[ u_j \ge \min \{ n \in \mathbb{N} : (n+1)^{p(\beta+1)-\beta} > j\} > j^{1/\{p(\beta+1)-\beta\}} -1. \]
Since $\displaystyle \sum_{j=1}^{\infty} \dfrac{1}{j^{1/\{p(\beta+1)-\beta\}}} <\infty$, we can find a positive constant $K=K(\mathbb{S})$ such that
\begin{align} \label{;1a+++}
\sum_{k=s_1}^n (\check{w}_k - w_k) = \sum_{k=s_1}^n \dfrac{1-x_{k+1}}{k}= \sum_{j=1}^{J_n} \dfrac{1}{u_j} \le K. 
\end{align}
Combining \eqref{;1b}--\eqref{;1a+++}, we obtain \eqref{;1a+}. This completes the proof.
\end{proof}



\begin{lemma} \label{lem:Sigma2Ratio} Let $p\in (0,1)$, $\beta \in (-1,p/(1-p))$. For any $\mathbb{S}$ as in \eqref{def:Subset} satisfying \eqref{;condA}, we have \\
(i) 
For $\beta \in (-1,0)$, there is a positive constant $K=K(p,\beta)$ not depending on $\mathbb{S}$ such that 
\begin{align*}
\dfrac{E^{\mathbb{S}}[\Xi_n^2]}{(E^{\mathbb{S}}[\Xi_n])^2} \le K
\text{ for $n\in\mathbb{N}$.}
 \end{align*}
(ii) 
Suppose $m\in \mathbb{N}$ with $s_m=N_0$, where $N_0$ is defined in \eqref{;condA}.
Put 
\begin{align}\label{:S_hat}
\hat{ \mathbb{S} } =\{ \hat{s}_i\}_{i\in\mathbb{N}} := \{ s_{m-1+i}\}_{i \in \mathbb{N}}.
\end{align}
For $\beta \in [0,p/(1-p))$, there is a positive constant $K=K(p,\beta)$ such that 
\begin{align*}
\dfrac{E^{\hat{\mathbb{S}}}[\Xi_n^2]}{ (E^{\hat{\mathbb{S}}}[\Xi_n])^2} \le K
\text{ for $n>\hat{s}_1$.}
\end{align*}
\end{lemma}

\begin{proof}
Note that calculations similar to those at the beginning of the proof of Theorem 2.1 give us
\begin{align*}
E^{\mathbb{S}}[\Xi_{n+1}^2 \mid \mathcal{F}_n] 
=\left(1+2p w_n \right) \cdot \Xi_n^2 + pw_n\mu_{n+1}\Xi_n.
\end{align*}
Setting $\tilde{L}_n := \Xi_n^2 / \prod_{\ell=s_1}^{n-1} \left(1+2p w_\ell \right)$,
we have
\begin{align*}
E^{\mathbb{S}}[\tilde{L}_{n+1}] - E^{\mathbb{S}}[\tilde{L}_n] =  \dfrac{pw_n\mu_{n+1} E^{\mathbb{S}}[\Xi_n]}{\prod_{\ell=s_1}^{n-1} \left(1+2p w_\ell \right)} {= pw_n\mu_{n+1} \mu_{s_1} \prod_{\ell=s_1}^{n-1} \dfrac{1+p w_\ell}{1+2p w_\ell}},  
\end{align*}
where we have used \eqref{;2g}. Then we see that
\begin{align*}
E^{\mathbb{S}}[\tilde{L}_n]
&=\mu_{s_1}^2+p\mu_{s_1} \sum_{k=s_1}^{n-1}  w_k\mu_{k+1} \prod_{\ell=s_1}^{k-1} \dfrac{1+p w_\ell}{1+2p w_\ell}.
\end{align*}
From \eqref{;2g}, 
\begin{align*}
 \dfrac{E^{\mathbb{S}} [\Xi_n^2]}{(E^{\mathbb{S}}[\Xi_n])^2} &= \prod_{j=s_1}^{n-1}  \frac{1+2pw_j}{(1+pw_j)^2} \cdot \dfrac{E^{\mathbb{S}}[\tilde{L}_n]}{\mu_{s_1}^2} \le \dfrac{E^{\mathbb{S}}[\tilde{L}_n]}{\mu_{s_1}^2} 
=1+ p \sum_{k=s_1}^{n-1} \frac{w_k\mu_{k+1}}{\mu_{s_1}}  \prod_{\ell=s_1}^{k-1} 
\frac{1+pw_\ell}{1+2pw_\ell}.
\end{align*}
(i) 
For $\beta \in (-1,0)$,
\begin{align*}
\sum_{k=s_1}^{n-1} \frac{ w_k\mu_{k+1}}{\mu_{s_1}} \prod_{\ell=s_1}^{k-1} \frac{1+pw_\ell}{1+2pw_\ell}
\le 
\frac{\beta+1}{\mu_{s_1}} \sum_{k=s_1}^{\infty} \frac{\mu_{k+1}}{k}  \to \dfrac{\beta+1}{-\beta} \mbox{ as $s_1\to \infty$,}
\end{align*}
where we have used \eqref{r-mu1}.
Thus we have (i).

\noindent (ii) 
For $\beta \in [0,p/(1-p))$ and $\hat{s}_1$ as in \eqref{:S_hat},
\begin{align}
\sum_{k=\hat{s}_1}^{n-1} \frac{\mu_{k+1}}{\mu_{\hat{s}_1}} w_k \prod_{\ell=\hat{s}_1}^{k-1} \frac{1+pw_\ell}{1+2pw_\ell} 
\le \dfrac{1}{\mu_{\hat{s}_1}}
\sum_{k=\hat{s}_1}^{n-1} \frac{\mu_{k+1}}{k}  \prod_{\ell=\hat{s}_1}^{k-1} \frac{1+pw_\ell}{1+2pw_\ell} .
 \label{:210a}
\end{align}
Since
\begin{align*}
\log \prod_{\ell=\hat{s}_1}^{k-1} \frac{1+pw_\ell}{1+2pw_\ell} &\le -p\sum_{\ell=\hat{s}_1}^{k-1}w_{\ell}+\dfrac{p^2}{2} \sum_{ \ell=\hat{s}_1}^{k-1} w_{\ell}^2 \\
&\le -p(\beta+1) \sum_{\ell=\hat{s}_1}^{k-1} \dfrac{x_{\ell+1}}{\ell} + \dfrac{p^2}{2} \sum_{ \ell=1}^{\infty} \dfrac{(\beta+1)^2}{\ell^2} \\
&= p(\beta+1) \sum_{\ell=\hat{s}_1}^{k-1} \dfrac{1-x_{\ell+1}}{\ell} - p(\beta+1) \sum_{\ell=\hat{s}_1}^{k-1} \dfrac{1}{\ell} + \dfrac{p^2}{2} \sum_{ \ell=1}^{\infty} \dfrac{(\beta+1)^2}{\ell^2},
\end{align*}
we have
\begin{align}
 \prod_{\ell=\hat{s}_1}^{k-1} \frac{1+pw_\ell}{1+2pw_\ell} &\le C_1 \left(\dfrac{k}{\hat{s}_1}\right)^{-p(\beta+1)} \exp \left(  p(\beta+1) \sum_{\ell=\hat{s}_1}^{k-1} \dfrac{1-x_{\ell+1}}{\ell} \right), \label{:210a+}
\end{align}
where $C_1=C_1(p,\beta)>0$ is a constant independent of $\mathbb{S}$.
We define 
\[ 
\mathcal{A} := \{ \{ u_j \} : \# ( \{ u_j \} \cap \{1,2,\ldots,n\} ) \leq n^{p(\beta+1)-\beta} \mbox{ for all $n \in \mathbb{N}$} \}. 
\]
Put $\{ u_j \} := \{ k : \hat{s}_1 < k,\,x_k=0 \}$ and $v_j := \min \{ n \in \mathbb{N} : n^{p(\beta+1)-\beta} \ge j \}$ for $j \in \mathbb{N}$.
Then$\{u_j\}, \{v_j\} \in \mathcal{A}$, and $v_j \leq u_j$ for $j \in \mathbb{N}$.  
Since $p(\beta+1)-\beta \in (0,1)$,
\[ 
\sum_{\ell=\hat{s}_1}^{k-1} \frac{ \mathbf{1}(x_{\ell+1}=0)}{\ell}=\sum_{\ell \in \{u_j\}} \dfrac{1}{\ell} \leq \sum_{\ell \in \{v_j\}} \dfrac{1}{\ell}  \le \sum_{j=1}^{\infty} \dfrac{1}{j^{1/\{p(\beta+1)-\beta\}}}< \infty.  
\]
Then from \eqref{:210a} and  \eqref{:210a+} we have
\[
\sum_{k=\hat{s}_1}^{n-1} \frac{\mu_{k+1}}{\mu_{\hat{s}_1}} w_k \prod_{\ell=\hat{s}_1}^{k-1} \frac{1+pw_\ell}{1+2pw_\ell} \le  \dfrac{C_2}{\mu_{\hat{s}_1}\hat{s}_1^{-p(\beta+1)}}
\sum_{k=\hat{s}_1}^{\infty} \frac{\mu_{k+1}}{k^{1+p(\beta+1)}} \to C_3 \mbox{ as $\hat{s}_1 \to \infty$,}
\]
where $C_2$ and $C_3$ are positive constants depending on $p$ and $\beta$, but independent of $\mathbb{S}$. This completes the proof.
\end{proof}


For the next result we need (see e.g. Lemma 4.14 in Stromberg \cite{Stromberg94}):

\begin{lemma}[the Paley--Zygmund inequality]\label{lem:unif}
Let $Z$ be a non-negative random variable satisfying
$E[Z] >0$ and $E[Z^2] <\infty$.
Then, for $\theta \in (0,1)$,  we have
\begin{align*}
P(Z > \theta E[Z]) \ge (1-\theta)^2 \cdot \dfrac{ (E[Z])^2}{E[Z^2]}.
\end{align*}
\end{lemma}

\begin{lemma} \label{lem:Sigma2Ratio2} Let $p \in (0,1)$. For all $\mathbb{S}$ as in \eqref{def:Subset} satisfying \eqref{;condA}, we have
 \\
(i) For $\beta \in (-1,0)$, there is a positive constant $K=K(p,\beta)$ not depending on $\mathbb{S}$ such that
$P^{\mathbb{S}}(T_n \asymp n^{p(\beta+1)-\beta}) \ge 1/K$.
\\
(ii) Let $\hat{\mathbb{S}}$ be as in \eqref{:S_hat}. For $\beta \in [0,p/(1-p))$, there is a positive constant $K=K(p,\beta)$ not depending on $\mathbb{S}$ such that 
$P^{\hat{\mathbb{S}}}(T_n \asymp n^{p(\beta+1)-\beta} ) \ge 1/K$. 
\end{lemma}

\begin{proof} By Lemmas \ref{lem:unif} and \ref{lem:Sigma2Ratio}, 
$P^{\mathbb{S}}\left(\Xi_n>E^{\mathbb{S}}[\Xi_n]/2 \right) \ge 1/K$.
This together with Lemma \ref{thm:XiSigmaMeanGen} gives
$P^{\mathbb{S}}(\Xi_n \asymp n^{p(\beta+1)}) \ge 1/K$. 
By the same argument as in the case $\mathbb{S}=\mathbb{N}$ we obtain
the conclusion. (See Lemma 4.2 in \cite{RoyTakeiTanemura24ECP}.)
\end{proof}

\section{Positivity of the martingale limit}
\label{sec:Postivity} 
For $j,k \in \mathbb{N}$ with $j <k$ we write $j \Leftarrow k$ if $\beta_k=j$, 
and $j \leftarrow k$ if there is an increasing sequence $\{ \ell_i\}_{i=0}^p$ of $\mathbb{N}$
with $\ell_0=j, \ell_p=k$ such that $\ell_i \Leftarrow \ell_{i+1}$, $i=0,1,\dots, p-1$.
Let (see Figure \ref{fig:coupling}):
\begin{align}
&\eta^{(0)}= \{1\},
\\
&\eta^{(1)}= \{i \in \mathbb{N} : \beta_i=1\} =:\{Y_j^{(1)} \}_{j=1}^{\# \eta^{(1)}},
\\
&\eta^{(m)}= \{i \in \mathbb{N} : k \Leftarrow i \text{ for some $k\in\eta^{(m-1)}$} \}=:\{Y^{(m)}_j\}_{j=1}^{\# \eta^{(m)}}, m\ge 2,
\end{align}
where $Y^{(m)}_j < Y^{(m)}_{j+1}$, $j\in\mathbb{N}$.
We set 
$\eta^{(m)}_n= \eta^{(m)}\cap \{1,2,\dots,n\}$.
We introduce another process defined as 
\begin{align}
\zeta^{(m,j)}&= \{Y_j^{(m)}\}\cup\{ i : Y_j^{(m)} \leftarrow i \}, \quad j =1,2,\dots, \# \eta^{(m)}, \\
\zeta_n^{(m,j)}&= \zeta^{(m,j)} \cap \{1,2,\ldots,n\},
\end{align}
and set
\begin{align}
\Lambda_j ^{(m)} = \begin{cases}
\biggl\{ \displaystyle \lim_{n\to\infty}\frac{\#\zeta_n^{(m,j)} }{n^{p(\beta+1) -\beta}}=0\biggr\}, &\text{if $j \le \# \eta^{(m)}$},
\\
\Omega, &\text{otherwise.}
\end{cases}
\end{align}
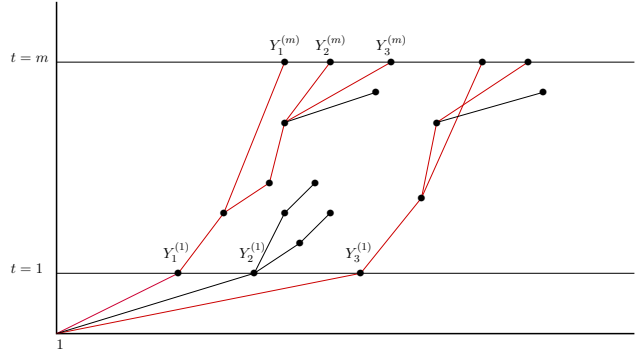
\begin{figure}
\begin{center}
\scalebox{.5} 
{
\begin{pspicture}(0,-4.6)(19.619728,4.6)
\definecolor{colour0}{rgb}{0.8,0.0,0.2}
\definecolor{colour1}{rgb}{0.8,0.0,0.0}
\psline[linecolor=black, linewidth=0.04](4.4197116,4.6)(4.4197116,-4.2)(19.61971,-4.2)
\psline[linecolor=black, linewidth=0.02](4.4197116,-2.6)(19.61971,-2.6)
\psline[linecolor=colour0, linewidth=0.02](4.4197116,-4.2)(7.6197114,-2.6)
\psline[linecolor=colour1, linewidth=0.02](4.4197116,-4.2)(12.419711,-2.6)
\psline[linecolor=black, linewidth=0.02](4.4197116,-4.2)(9.619712,-2.6)(9.619712,-2.6)
\rput[bl](7.2197113,-2.3){$Y_1^{(1)}$}
\rput[bl](9.219711,-2.3){$Y_2^{(1)}$}
\rput[bl](12.0197115,-2.3){$Y_3^{(1)}$}
\rput[bl](3.2197115,-2.6){$t=1$}
\psline[linecolor=black, linewidth=0.02](4.4197116,3.0)(19.61971,3.0)(19.61971,3.0)
\psline[linecolor=colour1, linewidth=0.02](7.6197114,-2.6)(8.819712,-1.0)(10.0197115,-0.2)(10.419711,1.4)(11.619712,3.0)
\psline[linecolor=colour1, linewidth=0.02](12.419711,-2.6)(14.0197115,-0.6)(14.419711,1.4)(16.819712,3.0)
\psline[linecolor=colour1, linewidth=0.02](15.619712,3.0)(14.0197115,-0.6)
\psline[linecolor=black, linewidth=0.02](9.619712,-2.6)(10.419711,-1.0)(11.219711,-0.2)
\psline[linecolor=black, linewidth=0.02](9.619712,-2.6)(10.819712,-1.8)(11.619712,-1.0)
\rput[bl](10.0197115,3.2){$Y_1^{(m)}$}
\psline[linecolor=colour1, linewidth=0.02](8.819712,-1.0)(10.419711,3.0)
\psline[linecolor=colour1, linewidth=0.02](10.419711,1.4)(10.419711,1.4)(10.419711,1.4)(13.219711,3.0)
\rput[bl](11.219711,3.2){$Y_2^{(m)}$}
\rput[bl](12.819712,3.2){$Y_3^{(m)}$}
\rput[bl](4.4197116,-4.6){1}
\rput[bl](3.2197115,3.0){$t=m$}
\psline[linecolor=black, linewidth=0.02](10.419711,1.4)(12.819712,2.2)
\psline[linecolor=black, linewidth=0.02](14.419711,1.4)(17.219711,2.2)
\psdots[linecolor=black, dotsize=0.04](0.019711532,-4.2)
\psdots[linecolor=black, dotsize=0.04](0.019711532,-4.2)
\psdots[linecolor=black, dotsize=0.04](0.019711532,-4.2)
\psdots[linecolor=black, dotsize=0.18](7.6197114,-2.6)
\psdots[linecolor=black, dotsize=0.18](8.819712,-1.0)
\psdots[linecolor=black, dotsize=0.18](10.0197115,-0.2)
\psdots[linecolor=black, dotsize=0.18](10.419711,1.4)
\psdots[linecolor=black, dotsize=0.18](12.819712,2.2)
\psdots[linecolor=black, dotsize=0.18](13.219711,3.0)
\psdots[linecolor=black, dotsize=0.18](11.619712,3.0)
\psdots[linecolor=black, dotsize=0.18](10.419711,3.0)
\psdots[linecolor=black, dotsize=0.18](11.219711,-0.2)
\psdots[linecolor=black, dotsize=0.18](10.419711,-1.0)
\psdots[linecolor=black, dotsize=0.18](11.619712,-1.0)
\psdots[linecolor=black, dotsize=0.18](10.819712,-1.8)
\psdots[linecolor=black, dotsize=0.18](9.619712,-2.6)
\psdots[linecolor=black, dotsize=0.18](12.419711,-2.6)
\psdots[linecolor=black, dotsize=0.18](14.0197115,-0.6)
\psdots[linecolor=black, dotsize=0.18](14.419711,1.4)
\psdots[linecolor=black, dotsize=0.18](17.219711,2.2)
\psdots[linecolor=black, dotsize=0.18](15.619712,3.0)
\psdots[linecolor=black, dotsize=0.18](16.819712,3.0)
\end{pspicture}
}
\caption{$\eta^{(m)}$ denotes all the integer points on the line $\{t=m\}$ which are eventually connected to the vertex $1$ on the $x$-axis via integer points on the levels $\{1 \leq t \leq m-1\}$. 
The black lines denote those which do not have any connection from $\{t \geq m\}$.
}   \label{fig:coupling}
    \end{center}
\end{figure}

We put $\xi = \{k\in\mathbb{N}: X_k=1\}$ and $\xi_n:=\xi \cap \{1,2,\dots,n\}$.
Then we have
\begin{align}
\xi = \eta^{(0)} \cup \eta^{(1)}\cup \Biggl\{ \bigcup_{j=1}^\infty \zeta^{(1,j)} \Biggr\}.
\end{align}
We note that $S_n = \# \xi_n$,
and
$\# \eta^{(1)} 
\begin{cases} =\infty &\text{a.s. if $\beta \in (-1, 0]$}
\\
<\infty &\text{a.s. if $\beta \in (0, \infty)$. }
\end{cases}$

First we consider the case $\beta\in (-1,0)$. 
Because $\displaystyle \lim_{n\to\infty}\frac{S_n}{n^{p(\beta+1) -\beta}}$ exists, 
\begin{align}
P\biggl( \lim_{n\to\infty}\frac{S_n}{n^{p(\beta+1) -\beta}} =0\biggr) 
&\le P\Biggl( \bigcap_{j=1}^\infty  \Lambda_j^{(1)} \Biggr)
\label{;3_8}
= P\left( \Lambda_1^{(1)} \right)\prod_{j=1}^\infty P\Biggl(\Lambda_{j+1}^{(1)} \,\left|\,  \bigcap_{\ell=1}^j  \Lambda_\ell^{(1)} \right. \Biggr).
\end{align}
Note that $k\notin \bigcup_{\ell=1}^j \zeta^{(1,\ell)}$ means $\beta_{k} \notin \bigcup_{\ell=1}^j \zeta^{(1,\ell)}$.
Then, for any $\mathbb{S}\subset \mathbb{N}$, $\beta_k$  under the conditional probability 
$P\left( \cdot \,\left|\,   \mathbb{N}\setminus \bigcup_{\ell=1}^j \zeta^{(1,\ell)} =  \mathbb{S}\right. \right) $
stochastically dominates  $\tilde{\beta}_k$ in \eqref{;beta_t} under $P^{\mathbb{S}}$ for any $k\in \mathbb{S}$.
Since $\beta_k, k\in\mathbb{S}$ under the conditional probability are independent, for any $A_i \subset \mathbb{S}$ and $\{ k_i\}_{i=1}^m  \in \mathbb{S}$, $m \in \mathbb{N}$, 
\begin{align}\label{;compa}
P\left(\beta_{k_i} \in A_i, \  1\le i \le m   \,\left|\,  
 \mathbb{N}\setminus\bigcup_{\ell=1}^j  \zeta^{(1,\ell)} =  \mathbb{S}\right. \right) 
\ge P^{\mathbb{S}}(\tilde{\beta}_{k_i} \in A_i, \  1\le i \le m ).
\end{align}
On the event $\bigcap_{\ell=1}^j \Lambda^{(1,\ell)}$,
$\mathbb{S}=\mathbb{N}\setminus \bigcup_{\ell=1}^j \zeta^{(1,\ell)}$ satisfies the condition \eqref{;condA}.
Then from Lemma \ref{lem:Sigma2Ratio2} (i) and \eqref{;compa},
\begin{align}\label{:225}
P\left(\Lambda_{j+1}^{(1)} \,\left| \,  \bigcap_{\ell=1}^j  \Lambda_\ell^{(1)}\right. \right) \le 1-\dfrac{1}{K},
\end{align}
and by \eqref{;3_8} we have 
$\displaystyle P\Bigl( \lim_{n\to\infty}\frac{S_n }{n^{p(\beta+1) -\beta}}=0\Bigr) =0$.

Next we consider the case $\beta\in [0, p/(1-p))$. In this case $P(S_\infty = \infty)>0$. 
Consider the conditional probability $P_\infty :=P( \,\cdot\, \mid S_\infty = \infty)$.
Note that 
\begin{align}\label{:key_b>0}
P\Bigl(\Bigl\{ \displaystyle \lim_{n\to\infty}\frac{S_n}{n^{p(\beta+1) -\beta}} =0\Bigr\} \cap \{ S_\infty =\infty \} \Bigr)=0
\end{align}
implies $\displaystyle P_\infty \Bigl(\lim_{n\to\infty}\frac{S_n}{n^{p(\beta+1) -\beta}} =0 \Bigr)=0$.

We first describe an algorithm (see Figure \ref{fig2}) to obtain a particular sequence of increasing integers $\{Y_{j_k}^{(m_k)}\}_{k \ge 0}$ such that $Y_{j_k}^{(m_k)} \leftarrow 1$ for $k \ge 0$ to use  Lemma \ref{lem:Sigma2Ratio2} (ii). \\
{\bf Step 0:} Fix $m_0 \in \mathbb{N}$, and take $\mathbb{S}_0 = \mathbb{N}$, and let $Y_{j_0}^{(m_0)} = Y_1^{(m_0)}$. \\
{\bf Step 1:} Take $\mathbb{S}_1 =  \mathbb{N} \setminus  \zeta^{(m_0,1)}$ and $N(1)=N_0(\mathbb{S}_1)$, where $N_0(\mathbb{S}_1)$ for this $\mathbb{S}_1$  is defined in \eqref{;condA}. Let
\begin{align*}
\tau_1^{(m_0)} := \inf \{ \ell > 1: \zeta^{(m_0,\ell)} \cap [N(1),\infty) \},
\end{align*}
where $\inf \emptyset = +\infty$. For $\tau_1^{(m_0)}<+\infty$, define 
\begin{align*}
Y_{j_1}^{(m_1)} = \min \{ \zeta^{(m_1,\tau_1^{(m_0)})} \cap [N(1),\infty)  \}.
\end{align*}
We stop the algorithm when $\tau_1^{(m_0)}=+\infty$. \\
{\bf Step $\pmb{k}$:} Take $\mathbb{S}_k =  \mathbb{N} \setminus  \zeta^{(m_0,\tau_{k-1}^{(m_0)})}$ and $N(k)=N_0(\mathbb{S}_k)$. Let
\begin{align*}
\tau_k^{(m_0)} := \inf \{ \ell > \tau_{k-1}^{(m_0)} : \zeta^{(m_0,\ell)} \cap [N(k),\infty) \},
\end{align*}
where $\inf \emptyset = +\infty$. For $\tau_k^{(m_0)}<+\infty$, define 
\begin{align*}
Y_{j_k}^{(m_k)} = \min \{ \zeta^{(m_0,\tau_{k-1}^{(m_0)})} \cap [N(k),\infty)  \}.
\end{align*}
We stop the algorithm when $\tau_k^{(m_0)}=+\infty$.

For $m \in \mathbb{N}$, we let
\begin{align}
\mathcal{C}^{(m)} = \{j \in \mathbb{N} : \# \zeta^{(m,j)}=\infty\}. 
\end{align}
Because $P_\infty( \# \mathcal{C}^{(m)} \le M, \ \forall m \in\mathbb{N})=0$ for any $M\in\mathbb{N}$,
\begin{align}\label{Cminfty}
P_\infty \left(\lim_{m\to\infty} \# \mathcal{C}^{(m)} = \infty\right)=1.
\end{align}
Note that  $\# \mathcal{C}^{(m)} \ge k+1$ implies $\tau_{k}^{(m)} < \infty$.

To apply the algorithm, take $m_0 \in \mathbb{N}$. 
Then from Lemma \ref{lem:Sigma2Ratio2} (ii) with $\hat{\mathbb{S}} = \hat{\mathbb{S}}_0 = \{k \in \mathbb{N} : k \ge Y_1^{(m_0)} \}$,
\begin{align}
P\left(\Lambda_{1}^{(m_0)} \right) \le 1-\dfrac{1}{K}.
\label{eq:3.15}
\end{align}
Take $\hat{\mathbb{S}}_1 = \mathbb{S}_1 \cap [Y_{j_1}^{(m_1)},\infty)$. On the event $\Lambda_{1}^{(m_0)}$, $\mathbb{S}_1$ satisfies \eqref{;condA} with $N_0=N(1)$. Applying Lemma \ref{lem:Sigma2Ratio2} (ii) for $\hat{\mathbb{S}}=\hat{\mathbb{S}}_1$,
\begin{align}
P\left(   \left. \Lambda_{j_1}^{(m_1)}   \,\right | \,  \Lambda_1^{(m_0)}, \  \tau_1^{(m_0)} <\infty \right) \le 1-\dfrac{1}{K}.
\label{eq:3.16}
\end{align}
Combining \eqref{eq:3.15} and \eqref{eq:3.16}, we obtain
\begin{align}
P\left(      \Lambda_1^{(m_0)} \cap  \Lambda_{j_1}^{(m_1)}, \  \tau_1^{(m_0)} <\infty   \right) \le \left(1-\dfrac{1}{K}\right)^2.
\end{align}
In view of \eqref{Cminfty}, iterating this procedure, we have
\begin{align}
P\Biggl( \bigcap_{p=0}^{\ell-1}  \Lambda_{j_p}^{(m_p)}, \ \tau_{\ell-1}^{(m_0)} <\infty   \Biggr) \le \left(1-\dfrac{1}{K}\right)^\ell \quad \mbox{for $\ell \in \mathbb{N}$}. 
\label{eq:3.14}
\end{align}
Hence
\begin{align}\notag
&P\left(\left\{ \displaystyle \lim_{n\to\infty}\frac{S_n}{n^{p(\beta+1) -\beta}} =0\right\} \cap \{ S_\infty =\infty \} \right)
\\\notag
&\le P\Biggl( \bigcap_{p=0}^{\ell-1}  \Lambda_{j_p}^{(m_p)}, \ \tau_{\ell-1}^{(m_0)} <\infty   \Biggr)
+ P(\tau_{\ell-1}^{(m_0)}=\infty, \ S_\infty =\infty )
\\
&\le \left(1-\dfrac{1}{K}\right)^\ell + P_{\infty} (\# \mathcal{C}^{(m_0)} < \ell),
\end{align}
where we used $\{ \tau_{\ell-1}^{(m_0)} = \infty \} \subset \{ \# \mathcal{C}^{(m_0)} < \ell \}$.
For any $\varepsilon >0$ we take $\ell \in \mathbb{N}$ such that $\{1-(1/K)\}^{\ell} < \varepsilon/2$, 
and then
we can take $m_0 \in \mathbb{N}$ satisfying $ P_{\infty} (\# \mathcal{C}^{(m_0)} < \ell)< \varepsilon/2$ by \eqref{Cminfty}. Thus, we have \eqref{:key_b>0}.
This completes the proof of our assertions of first two rows given in Table \ref{table2}.
\qed

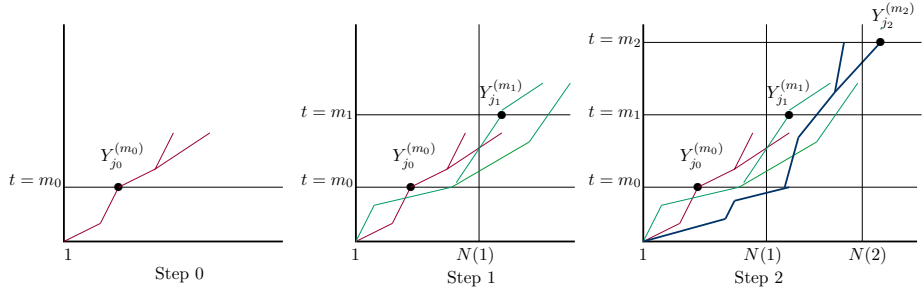
\begin{figure}
\begin{center}
\scalebox{.6} 
{
\begin{pspicture}(0,-3.105)(20.0,3.105)
\definecolor{colour3}{rgb}{0.0,0.6,0.6}
\definecolor{colour4}{rgb}{0.6,0.0,0.2}
\definecolor{colour5}{rgb}{0.0,0.6,0.2}
\definecolor{colour2}{rgb}{0.0,0.6,0.4}
\definecolor{colour6}{rgb}{0.0,0.2,0.4}
\psline[linecolor=colour3, linewidth=0.04](19.1,2.595)(19.1,2.595)
\psline[linecolor=black, linewidth=0.04](1.2,2.695)(1.2,-2.105)(6.0,-2.105)(6.0,-2.105)
\rput[bl](12.7,2.195){$t=m_2$}
\rput[bl](0.0,-0.905){$t=m_0$}
\psline[linecolor=black, linewidth=0.02](1.2,-0.905)(6.0,-0.905)
\psline[linecolor=colour4, linewidth=0.02](1.2,-2.105)(2.0,-1.705)
\psline[linecolor=colour4, linewidth=0.02](2.4,-0.905)(2.0,-1.705)
\psdots[linecolor=black, dotsize=0.18](2.4,-0.905)
\psline[linecolor=colour4, linewidth=0.02](2.4,-0.905)(3.2,-0.505)
\psline[linecolor=colour4, linewidth=0.02](3.6,0.295)(3.2,-0.505)(4.4,0.295)
\rput[bl](18.9,2.595){$Y^{(m_2)}_{j_2}$}
\rput[bl](2.0,-0.505){$Y^{(m_0)}_{j_0}$}
\rput[bl](3.2,-3.005){Step 0}
\rput[bl](1.2,-2.505){$1$}
\psline[linecolor=black, linewidth=0.04](7.6,2.695)(7.6,-2.105)(12.4,-2.105)(12.4,-2.105)
\rput[bl](6.4,-0.905){$t=m_0$}
\psline[linecolor=black, linewidth=0.02](7.6,-0.905)(12.4,-0.905)
\psline[linecolor=colour4, linewidth=0.02](7.6,-2.105)(8.4,-1.705)
\psline[linecolor=colour4, linewidth=0.02](8.8,-0.905)(8.4,-1.705)
\psdots[linecolor=black, dotsize=0.18](8.8,-0.905)
\psline[linecolor=colour4, linewidth=0.02](8.8,-0.905)(9.6,-0.505)
\psline[linecolor=colour4, linewidth=0.02](10.0,0.295)(9.6,-0.505)(10.8,0.295)
\rput[bl](8.4,-0.505){$Y^{(m_0)}_{j_0}$}
\rput[bl](9.8,-2.605){$N(1)$}
\rput[bl](7.5,-2.505){$1$}
\rput[bl](18.4,-2.605){$N(2)$}
\psline[linecolor=black, linewidth=0.02](10.3,2.695)(10.3,-2.105)
\psline[linecolor=colour5, linewidth=0.02](9.7,-0.905)(11.4,0.095)
\psline[linecolor=black, linewidth=0.02](7.6,0.695)(12.3,0.695)
\psline[linecolor=colour2, linewidth=0.02](7.6,-2.105)(8.0,-1.305)
\psline[linecolor=colour2, linewidth=0.02](8.0,-1.305)(9.7,-0.905)(9.7,-0.905)
\psline[linecolor=colour2, linewidth=0.02](9.8,-0.805)(10.8,0.695)
\psdots[linecolor=black, dotsize=0.18](10.8,0.695)
\psline[linecolor=colour2, linewidth=0.02](10.8,0.795)(11.7,1.395)
\psline[linecolor=colour2, linewidth=0.02](11.4,0.095)(12.3,1.395)
\rput[bl](10.3,0.895){$Y^{(m_1)}_{j_1}$}
\rput[bl](9.6,-3.105){Step 1}
\rput[bl](6.4,0.595){$t=m_1$}
\psline[linecolor=black, linewidth=0.04](13.9,2.695)(13.9,-2.105)(18.7,-2.105)(18.7,-2.105)
\rput[bl](12.7,-0.905){$t=m_0$}
\psline[linecolor=black, linewidth=0.02](13.9,-0.905)(18.7,-0.905)
\psline[linecolor=colour4, linewidth=0.02](13.9,-2.105)(14.7,-1.705)
\psline[linecolor=colour4, linewidth=0.02](15.1,-0.905)(14.7,-1.705)
\psdots[linecolor=black, dotsize=0.18](15.1,-0.905)
\psline[linecolor=colour4, linewidth=0.02](15.1,-0.905)(15.9,-0.505)
\psline[linecolor=colour4, linewidth=0.02](16.3,0.295)(15.9,-0.505)(17.1,0.295)
\rput[bl](14.7,-0.505){$Y^{(m_0)}_{j_0}$}
\rput[bl](16.1,-2.605){$N(1)$}
\rput[bl](13.8,-2.505){$1$}
\psline[linecolor=black, linewidth=0.02](16.6,2.695)(16.6,-2.105)
\psline[linecolor=colour5, linewidth=0.02](16.0,-0.905)(17.7,0.095)
\psline[linecolor=black, linewidth=0.02](13.9,0.695)(18.6,0.695)
\psline[linecolor=colour2, linewidth=0.02](13.9,-2.105)(14.3,-1.305)
\psline[linecolor=colour2, linewidth=0.02](14.3,-1.305)(16.0,-0.905)(16.0,-0.905)
\psline[linecolor=colour2, linewidth=0.02](16.1,-0.805)(17.1,0.695)
\psdots[linecolor=black, dotsize=0.18](17.1,0.695)
\psline[linecolor=colour2, linewidth=0.02](17.1,0.795)(18.0,1.395)
\psline[linecolor=colour2, linewidth=0.02](17.7,0.095)(18.6,1.395)
\rput[bl](16.6,0.895){$Y^{(m_1)}_{j_1}$}
\rput[bl](15.9,-3.105){Step 2}
\rput[bl](12.7,0.595){$t=m_1$}
\psline[linecolor=black, linewidth=0.02](13.9,2.295)(20.0,2.295)
\psline[linecolor=black, linewidth=0.02](18.6,0.695)(20.0,0.695)
\psline[linecolor=black, linewidth=0.02](18.7,-0.905)(19.9,-0.905)
\psline[linecolor=black, linewidth=0.04](18.7,-2.105)(19.9,-2.105)
\psline[linecolor=colour6, linewidth=0.04](13.9,-2.105)(15.7,-1.605)
\psline[linecolor=colour6, linewidth=0.04](15.7,-1.605)(15.7,-1.605)(15.7,-1.605)(15.9,-1.205)
\psline[linecolor=colour6, linewidth=0.04](15.9,-1.205)(17.1,-0.905)
\psline[linecolor=colour6, linewidth=0.04](17.0,-0.905)(17.0,-0.905)(17.3,0.195)
\psline[linecolor=colour6, linewidth=0.04](17.3,0.195)(18.1,1.195)
\psline[linecolor=colour6, linewidth=0.04](18.1,1.195)(18.3,2.295)(18.3,2.295)
\psline[linecolor=colour6, linewidth=0.04](18.1,1.195)(19.1,2.295)
\psdots[linecolor=black, dotsize=0.18](19.1,2.295)
\psline[linecolor=black, linewidth=0.02](18.7,2.695)(18.7,-2.105)
\end{pspicture}
}
\caption{The first three steps of the algorithm.}
    \label{fig2}
    \end{center}
\end{figure}


\section{Central limit theorems} 
\label{sec:CLT}

To prove Theorem \ref{:Theorem 3.1}  we first rewrite $S_n$ as follows:
\begin{align*}
S_n 
= \sum_{k=1}^n \dfrac{\Sigma_k-\Sigma_{k-1}}{\mu_k} 
= \dfrac{\Sigma_n}{\mu_n} + \sum_{k=1}^{n-1} \left( \dfrac{1}{\mu_k}- \dfrac{1}{\mu_{k+1}}\right) \Sigma_k
=\dfrac{\Sigma_n}{\mu_n} + \sum_{k=1}^{n-1}\dfrac{\beta}{k\cdot \mu_{k+1}} \cdot \Sigma_k.
\end{align*}
Noting that $E[X_{k+1} \mid \mathcal{F}_k]= \dfrac{p(\beta+1)}{k \cdot \mu_{k+1}}\Sigma_k$ for $k \ge 1$,  we have
\begin{align*}
S_n &= \dfrac{\Sigma_n}{\mu_n} + \dfrac{\beta}{p(\beta+1)}  \sum_{k=1}^{n-1} E[X_{k+1} \mid \mathcal{F}_k] = \dfrac{\Sigma_n}{\mu_n} + \dfrac{\beta}{p(\beta+1)} \sum_{j=1}^n (X_j-d_j) \\
&= \dfrac{\Sigma_n}{\mu_n} + \dfrac{\beta}{p(\beta+1)}  S_n -  \dfrac{\beta}{p(\beta+1)}  \sum_{j=1}^n d_j,
\end{align*}
where we put $d_1:=X_1=1$ and $d_j := X_j -E[X_j \mid \mathcal{F}_{j-1}]$ for $j \ge 2$. Using \eqref{def:MnMart} and \eqref{eq:DefCbeta}, and recalling that $\mu_n=c_n(\beta)$, we have
\begin{align*}
S_n = C(p,\beta) \cdot \dfrac{\Gamma(n+p(\beta+1))}{\Gamma(n+\beta)} \cdot M_n -  \dfrac{\beta}{p(\beta+1)-\beta}  \sum_{j=1}^n d_j.
\end{align*}
Let $\hat{d}_k=M_k-M_{k-1}$. Note that $\widehat{d}_1=1$, and for $k \ge 2$,
\begin{align*}
\hat{d}_k
& = \frac{\Sigma_k- (1+\frac{p(\beta+1)}{k-1})\Sigma_{k-1} }{c_k(p(\beta+1))} 
= \frac{\mu_kX_k- E[\mu_k X_k| \mathcal{F}_{k-1}]}{c_k(p(\beta+1))}
= \frac{c_k(\beta)}{c_k(p(\beta+1))}d_k.
\end{align*}
Now we look at 
\begin{align*}
\dfrac{S_n - C(p,\beta)M_{\infty}n^{p(\beta+1)-\beta}}{\sqrt{n^{p(\beta+1)-\beta}}} =  -\sum_{k=1}^{\infty} X_{n,k} + R_n,
\end{align*}
where $(X_{n,k})_{k \ge 1,\,n \ge 1}$ is a square integrable martingale difference array defined by
\begin{align*}
X_{n,k}:= \begin{cases}
\displaystyle{
\frac{\beta}{ (p(\beta+1)-\beta) \sqrt{ n^{ p(\beta+1)-\beta} } }  d_k,
}
& 1 \le k \le n
\\
\displaystyle{
\frac{C(p,\beta)}{\sqrt{ n^{p(\beta+1)-\beta} }}\frac{\Gamma(n+p(\beta+1))}{\Gamma (n+\beta)}\hat{d}_k,
}
&k \ge n+1, 
\end{cases}
\end{align*}
and 
$R_n := \dfrac{C(p,\beta)}{\sqrt{ n^{p(\beta+1)-\beta} }} \cdot \left\{ \dfrac{\Gamma(n+p(\beta+1))}{\Gamma(n+\beta)} -n^{p(\beta+1)-\beta} \right\} \cdot M_{\infty}$.

Let $x>0$. Wendel's inequality \cite{Wendel48} implies that
\begin{align}\label{Est:1}
x^{\alpha}+ \alpha(\alpha-1) x^{\alpha-1}
\leq \dfrac{\Gamma(x+\alpha)}{ \Gamma(x)} \leq x^{\alpha}
\quad \mbox{for $\alpha \in [0,1]$. }
\end{align}
As for $\alpha>1$, letting $\{\alpha\}$ denote the fractional part of $\alpha$ and $k=\alpha-\{\alpha\}$, we have 
\begin{align} \label{Est:2pre}
x^k \cdot \frac{\Gamma(x+\{ \alpha \})}{\Gamma (x)} \le \frac{\Gamma(x+\alpha)}{\Gamma (x)} 
\le (x+\alpha-1)^k \cdot \frac{\Gamma(x+\{\alpha\})}{\Gamma (x)}.
\end{align} 
This together with \eqref{Est:1} implies that
\begin{align} \label{Est:2}
x^{\alpha} + \{\alpha\} (\{ \alpha \}-1) x^{\alpha-1} \le \dfrac{\Gamma(x+\alpha)}{ \Gamma(x)} \le (x+\alpha-1)^{\alpha} 
\quad \mbox{for $\alpha >1$. }
\end{align}
By \eqref{Est:1} and \eqref{Est:2}, we can see that
$\dfrac{\Gamma(n+p(\beta+1))}{\Gamma(n+\beta)} -n^{p(\beta+1)-\beta} = \mathcal{O} (n^{p(\beta+1)-\beta-1})$.
Thus $R_n \to 0$ as $n \to \infty$ a.s.

For random variables $(Z_n)_{n\in\mathbb{N}}$ and $Z$ defined on a probability space $(\Omega, \mathcal{H}, P)$
and $\mathcal{G}$ is a sub-$\sigma$-field of $\mathcal{H}$,  
we say that $(Z_n)_{n\in\mathbb{N}}$ converges $\mathcal{G}$-stably to $Z$ as $n\to\infty$,
written as
$Z_n \to Z$ $\mathcal{G}$-stably as $n\to\infty$, 
if $Z_n \to Z$ in distribution under $P(\,\cdot\, \mid F)$ for every $F\in \mathcal{G}$ with $P(F)>0$.
Then Theorem \ref{:Theorem 3.1} is derived by applying the following martingale CLT, which is Exercise 6.2 based on Theorem 6.1 in H\"{a}usler and Luschgy \cite{HauslerLuschgy15}, p.86, with $\mathcal{G}_{n,k} = \mathcal{F}_k$.

\begin{theoremN}[H\"{a}usler and Luschgy] \label{thm:HauslerLuschgy15Thm6.1} 
Let $(X_{n.k})_{0 \le k < \infty, n \in \mathbb{N}}$ be a square-integrable martingale difference array adapted to the nested array $(\mathcal{G}_{n,k})_{0 \le k < \infty, n \in \mathbb{N}}$.
We assume that
\[ \sum_{k=1}^{\infty} X_{n,k}
\mbox{ converges a.s. and }  
\sum_{k=1}^{\infty} E[X_{n,k}^2 \mid \mathcal{G}_{n,k-1}]<+\infty
\mbox{ a.s. for each $n$.} \]
Let $\mathcal{G} =\sigma\left( \bigcup_{n=1}^\infty \bigcup_{k=1}^{\infty} \mathcal{G}_{n,k} \right)$.
Assume that
\begin{align}
\label{:N} 
&\sum_{k=1}^{\infty} E[X_{n,k}^2 \mid \mathcal{G}_{n, k-1}] \to \eta^2
\quad \text{in probability as $n\to\infty$}
\intertext{for some $\mathcal{G}$-measurable real random variable $\eta\ge 0$, and}
\label{:CLB}
&\sum_{k=1}^{\infty} E[X_{n,k}^2 \mathbf{1}_{ \{ |X_{n,k}| \ge \varepsilon \} } \mid \mathcal{G}_{n, k-1}] \to 0
\quad \text{in probability as $n\to\infty$,}
\end{align}
 for every $\varepsilon  >0$. 
Then we have 
$\sum_{k=1}^{\infty} X_{n,k} \to \eta \cdot N$ $\mathcal{G} $-stably as $n \to \infty$, 
where $N$ is independent of $\mathcal{G}$ and $N \stackrel{d}{=} N(0,1)$. If $P(\eta>0)>0$ in addition, then
\begin{align*}
\bigg(   \sum_{k=1}^{\infty} E[ X_{n,k}^2 \mid \mathcal{G}_{n,k-1}] \bigg)^{-1/2} \sum_{k=1}^{\infty} X_{n,k} \to N  \quad \mbox{ $\mathcal{G}$-stably under $P(\,\cdot \, \mid \eta>0)$ as $n\to\infty$.}
\end{align*}
\end{theoremN}

 \medskip
Recalling that
$E[X_{n+1} \mid \mathcal{F}_n] = \dfrac{p(\beta+1)}{n\mu_{n+1}} \cdot \Sigma_n = \dfrac{p(\beta+1) c_n(p(\beta+1))}{n  c_{n+1}(\beta)} \cdot M_n$,
we have 
\begin{align}
E[X_{n+1} \mid \mathcal{F}_n] &\sim \dfrac{\Gamma(\beta+1)}{\Gamma(p(\beta+1))} \cdot n^{p (\beta+1)-\beta-1}\cdot M_{\infty}
\quad \mbox{as $n \to \infty$ a.s. on $\{M_{\infty}>0\}$}.
\label{X_n+1Asymp}
\end{align}
Noting that $p(\beta+1)-\beta \in (0,1)$, we have
\begin{align}
E[d_k^2 \mid \mathcal{F}_{k-1}] 
= E[ X_k \mid \mathcal{F}_{k-1}] \cdot (1-E[ X_k \mid \mathcal{F}_{k-1}]) 
&\sim E[ X_k \mid \mathcal{F}_{k-1}]
 \quad \mbox{as $k \to \infty$ a.s.}
\label{eq:Edk2asymp}
\end{align}
%
From \eqref{X_n+1Asymp} and \eqref{eq:Edk2asymp}, we have a.s. on $\{M_{\infty}>0\}$,
\begin{align}
\lim_{n\to\infty} \sum_{k=1}^n E[X_{n,k}^2  \mid  \mathcal{F}_{k-1}] &=  \lim_{n\to\infty}  \frac{\beta^2}{(p(\beta+1)-\beta)^2} \frac{1}{ n^{p(\beta+1)-\beta}} \sum_{k=1}^n  E[d_k^2 \mid \mathcal{F}_{k-1}] \notag \\
&=\frac{\beta^2}{(p(\beta+1)-\beta)^3} \cdot \frac{\Gamma(\beta+1)}{\Gamma(p(\beta+1))} \cdot M_\infty \notag \\
&=\frac{\beta^2}{(p(\beta+1)-\beta)^2} \cdot C(p,\beta)\cdot M_\infty, 
\label{;N_1}
\end{align}
and
\begin{align}
&\lim_{n\to\infty} \sum_{k=n+1}^\infty E[X_{n,k}^2  \mid  \mathcal{F}_{k-1}] \notag \\
&=\lim_{n\to\infty} \frac{C(p,\beta)^2}{n^{p(\beta+1)-\beta}} \frac{\Gamma(n+p(\beta+1))^2}{\Gamma (n+\beta)^2} \sum_{k=n+1}^\infty \frac{c_k (\beta)^2}{c_k (p(\beta+1))^2} \cdot E[d_k^2 \mid \mathcal{F}_{k-1}] \notag \\
&=\frac{(p(\beta+1))^2}{(p(\beta+1)-\beta)^3} \cdot \frac{\Gamma(\beta+1)}{\Gamma(p(\beta+1))} \cdot M_\infty =\frac{(p(\beta+1))^2}{(p(\beta+1)-\beta)^2} \cdot C(p,\beta) \cdot M_\infty.  \label{;N_2}
\end{align}
From \eqref{;N_1} and \eqref{;N_2} we have \eqref{:N} on $\{M_{\infty}>0\}$.
We can readily have  \eqref{;N_2} on $\{M_{\infty} = 0\}$.

Because \eqref{:CLB} on $\{M_{\infty} = 0\}$ is derived from \eqref{:N},
we show \eqref{:CLB}  on $\{M_{\infty} > 0\}$.  Using $|d_k| \le 1$ and \eqref{X_n+1Asymp},
we see that there exists a positive random variable $D_1$ independent of $k$ such that 
\begin{align}\label{:d4}
E[d_k^4  \mid  \mathcal{F}_{k-1}] \le E[d_k^2  \mid  \mathcal{F}_{k-1}] \le D_1 k^{p(\beta+1)-\beta-1}  \quad\mbox{on $\{M_{\infty}>0\}$}.
\end{align}
Hence, there exists a positive random variable $D_2$ independent of $n$,
\begin{align}
\sum_{k=1}^n E[ X_{n,k}^4 \mid  \mathcal{F}_{k-1}] &=\frac{\beta^4 }{(p(\beta+1)-\beta)^4} \dfrac{1}{n^{2( p(\beta+1)-\beta)}} \sum_{k=1}^n E[d_k^4 \mid  \mathcal{F}_{k-1}] \notag \\
&\le D_2 n^{-( p(\beta+1)-\beta)} 
\quad \mbox{on $\{M_{\infty}>0\}$}.\label{;X4_N}
\end{align}
From \eqref{:d4}, there exists a positive random variable $D_3$ independent of $n$,
\begin{align}
\sum_{k=n+1}^\infty E[ X_{n,k}^4 \mid  \mathcal{F}_{k-1}] 
&=\frac{C(p,\beta)^4}{n^{2(p(\beta+1)-\beta)}} \frac{\Gamma(n+p(\beta+1))^4}{\Gamma (n+\beta)^4} \sum_{k=n+1}^\infty \frac{c_k (\beta)^4}{c_k (p(\beta+1))^4} \cdot E[d_k^4 \mid \mathcal{F}_{k-1}] \notag \\
&\le D_3 n^{-( p(\beta+1)-\beta)} \quad \mbox{on $\{M_{\infty}>0\}$}.\label{;X4_N+1}
\end{align}
Combining \eqref{;X4_N} and \eqref{;X4_N+1}, we have $\sum_{k=1}^\infty E[ X_{n,k}^4 \mid  \mathcal{F}_{k-1}] \le (D_2 + D_3) n^{-( p(\beta+1)-\beta)}$ on $\{M_{\infty}>0\}$.
Because
$E[ X_{n,k}^2 \mathbf{1}_{\{ |X_{n,k}|\ge \varepsilon \}} \mid  \mathcal{F}_{k-1}] \le \varepsilon^{-2} \cdot E[ X_{n,k}^4\mid  \mathcal{F}_{k-1}]$,
we have \eqref{:CLB}.
\qed

\section*{Acknowledgements}

The problem studied here arose from a question asked by the referee of our earlier paper \cite{RoyTakeiTanemura24ECP}. 
R.R. thanks Keio University for its hospitality during multiple visits. M.T. and H.T. thank the Indian Statistical Institute for its hospitality during multiple visits. M.T. is partially supported by JSPS KAKENHI Grant Numbers JP19K03514 and JP22K03333. H.T. is partially supported by JSPS KAKENHI Grant Number JP19K03514, JP21H04432 and JP23H01077.


\end{document}